\documentclass{amsart}%
\usepackage{amsfonts}%
\usepackage{amsmath}%
\usepackage{amssymb}%
\usepackage{graphicx}
\providecommand{\U}[1]{\protect\rule{.1in}{.1in}}
\newtheorem{theorem}{Theorem}
\theoremstyle{plain}

\newtheorem{corollary}{Corollary}

\newtheorem{definition}{Definition}
\newtheorem{example}{Example}

\newtheorem{lemma}{Lemma}

\newtheorem{proposition}{Proposition}

\numberwithin{equation}{section}
\begin{document}
\title[Stability of invariant measures]{Stability of invariant measures }
\author{Sini\v{s}a Slijep\v{c}evi\'{c}}
\address{Department of Mathematics, Bijeni\v{c}ka 30, Zagreb,\ Croata}
\email{slijepce@math.hr}
\urladdr{}
\date{January 4, 2008}

\begin{abstract}
We generalize various notions of stability of invariant sets of dynamical
systems to invariant measures, by defining a topology on the set of
measures.\ The defined topology is similar, but not topologically equivalent
to weak* topology, and it also differs from topologies induced by the Riesz
Representation\ Theorem. It turns out that the constructed topology is a
solution of a limit case of a $p$-optimal transport problem, for $p=\infty$.

\end{abstract}
\maketitle

\section{Introduction}

The motivation for this paper is two-fold. The first motivation is related to
dynamical systems, and to finding the "right" topology on the set of measures,
namely a topology which replicates certain properties of a dynamical system on
a metric space to the induced dynamical system on the set of measures. The
second motivation is to investigate an alternative formulation of the
well-known Optimal transport problem. \vspace{2ex}

{\ }\textbf{The "right" topology on the set of measures. }Let $f$ be a
continuous function (i.e. a discrete dynamical system) on a metrizable
topological space $X$, and $f_{\sharp}$ be the induced function on the set of
Borel probability measures $P(X)$ on $X$. Then $f_{\sharp}$ is a discrete
dynamical system on $P(X)$. Study of the dynamical system $f_{\sharp}$ can
often give useful information on the system $f$, as was shown for example in
\cite{Gallay:01}, \cite{Slijep:00}. Now, a number of fundamental properties of
a dynamical system $f$ are not analogous to similar properties of the
dynamical system $f_{\sharp}$, in any of the standard topologies, such as the
weak* topology on $P(X)$. We give two examples.

Let $x\in X$ be a sink or a source of $f$, and let $\delta_{x}$ be the
probability measure supported on $\{x\}$. Then, typically, $\delta_{x}$ is
\textit{not} a sink or a source of $f_{\sharp}$ in the weak* topology. A
detailed discussion of this is in Section 5.

Now, let $A\subset X$ be a closed invariant set of $f$, and in addition an
attracting set. Let $A_{\sharp}$ be the set of all measures in $P(X)$ such
that its support is a subset of $A$. One can see (and we show in Section 5)
that $A_{\sharp}$ is an invariant set of $f_{\sharp}$, but is typically not an
attractor with respect to the weak* topology.

The same conclusions as in the last two examples hold for flows and semiflows.
They also hold for other usual topologies on the set of measures, induced by
topologies on the set $C^{\ast}(X)$ of all bounded linear functionals on the
set of continuous real valued functions $C(X)$, and the Riesz Representation Theorem.

Our goal is to find the "right" topology on the set of measures which would
naturally generalize various notions of stability and attraction to dynamical
systems on the set of measures. The topology should be close enough to the
weak* topology, so that one can find nontrivial compact sets of measures, and
use it in various applications.\ Such a topology constructed here is called
\emph{dynamical topology}. We will also show that, by identifying stable or
attracting sets of measures (with respect to the dynamical topology) rather
than sets of points, we get better insight and more information on behavior of
a chosen dynamical system.\vspace{2ex}

{\ \textbf{The $\infty$-optimal transport problem.}} Let $X$ be a metric space
with a metric $d$, and $\mu$, $\nu$ be two Borel probability measures on $X$.
Then the $p$-Wasserstein distance $W_{p}$, where $1\leq p<\infty$, is the
function
\begin{align}
C_{p}^{p}(\gamma)  & =\int_{X\times X}d(x,y)^{p}d\gamma(x,y),\label{e1:wass0}%
\\
W_{p}(\mu,\nu)  & =\inf\{C_{p}(\gamma)\,|\,\gamma\in T(\mu,\nu
)\}.\label{e1:wass}%
\end{align}
The set $T(\mu,\nu)$ is the set of all \textit{transports}; i.e. the set of
all Borel probability measures $\gamma$ on $X\times X$, such that
$\pi_{1\sharp}\gamma=\mu$, $\pi_{2\sharp}\gamma=\nu$, where $\pi_{1}$,
$\pi_{2}$ are projections of $X\times X$ to the first, resp. second variable
(transports are sometimes called "couplings" in the probability and some
ergodic theory literature).

The measure $\gamma$ which minimizes (\ref{e1:wass}) is a solution of the
optimal transport problem in the $p$-norm. Intuitively, the minimizer $\gamma$
is the measure which describes how the points in the support of $\mu$ are
coupled to the points in the support of $\nu$, so that the $p$-norm of the
coupling distances is minimal.

There is rich literature on the optimal transport problem, including proofs of
existence, uniqueness, and properties for various spaces $X$ and norms $1\leq
p<\infty$ (see e.g. \cite{Ambrosio:05}, \cite{Carlen:03}, \cite{Fragala:05},
\cite{Pratelli:05}). (The metric $d(x,y)$ can also be replaced by a more
general cost function $c(x,y)$.)

In this paper we study the case $p=\infty$. This case may have implications to
various optimization problems where the cost of transport does not depend on
mass to be transported, but only on the maximal transport distance.

For all $1\leq p<\infty$ the formula (\ref{e1:wass}) generates a metric on the
space of the probability measures $P(X)$, called p-Wasserstein metric. One can
show that p-Wasserstein metrics are for $1\leq p<+\infty$ uniformly equivalent
to each other and to the Prokhorov metric, and so generate the weak* topology
on $P(X)$ (see e.g. \cite{Gibbs:02}). We will show that they are neither
uniformly nor topologically equivalent to the metric and topology in the case
$p=+\infty$, and that the topology in the case $p=+\infty$ is the dynamical
topology defined in the first part of the paper.\vspace{2ex}

{\ \textbf{The structure of the paper.}} We start with definitions of
\emph{dynamical metric} and \emph{dynamical topology}, which are the main
tools of this paper. Then we show that the dynamical topology indeed differs
from the weak* topology, and from topologies induced by the Riesz
Representation\ Theorem. In the third section, we discuss various properties
of the dynamical topology. We focus on characterization of convergence of the
set of measures with respect to the dynamical topology in the fourth
section.\ The proof of characterization of convergence is combinatorial in
character (as such, it could have implications in the combinatorial ergodic
theory). We continue with proofs that the dynamical topology indeed gives
natural generalizations of various notions of stability to dynamical systems
on sets of measures. Finally we analyze the $\infty$-optimal transport
problem, and we show that the $\infty$-Wasserstein metric generated by the
$\infty$-optimal transport problem generates the dynamical topology on the
$P(X)$. We also prove existence of a solution of the $\infty$-optimal
transport problem.

\section{Definition of the dynamical topology}

In this paper $X$ is always a compact metric space, equipped with a Borel
$\sigma$-algebra into a measurable space, and $d$ is its metric. Let $M(X)$ be
the space of all finite Borel measures on $X$, and $P(X)\subset M(X)$ the
space of all probability measures. Let $C(X)$ be the normed (Banach) space of
all real valued continuous functions $f:X\rightarrow{\mathbb{R}}$. Then $M(X)$
can be naturally embedded in the dual space of all bounded linear functionals
on $C(X)$.

We denote by $\tau_{w}$ the weak* topology $M(X)$, and by $\tau_{u}$ the
uniform topology (i.e. the topology induced by the $\sup$-norm on the dual
$C^{\ast}(X)$ of $C(X)$). The topology $\tau_{u}$ is much finer than the
topology $\tau_{w}$, and as such is seldom used in the dynamical systems. We
will show that both topologies differ from the $\tau_{d}$ topology to be
constructed. We use the notation "w-", "u-", and "d-" ("d"\ for the
\emph{dynamical topology}, yet to be defined) when referring to properties of
a set or a sequence in various topologies. In particular, \emph{w-convergent},
\emph{u-convergent}, and \emph{d-convergent} means that a sequence of measures
in $P(X)$ is convergent with respect to weak*, uniform, or dynamical topology respectively.

We now define dynamical metric and topology on the set $P(X)$. Let $I$ be the
unit interval $[0,1]$, and $\lambda$ the Lebesgue measure defined on the
family of Borel-measurable subsets of $I$. Given two functions
$f,g:I\rightarrow X$, we define their distance as
\[
D(f,g)=\sup_{a\in I}d(f(a),g(a)).
\]
Distance $D$ is well defined because of compactness of $X$. It is
straightforward to check that $D$ is symmetric, and that it satisfies the
triangle inequality $D(f,g)+D(g,h)\geq D(f,h)$.

If $f:I\rightarrow X$ is a (Borel)\ measurable function, then $f_{\sharp
}\lambda=\mu$ denotes the measure $\mu(A)=\lambda(f^{-1}(A))$ for all
measurable $A\subseteq X$.

\begin{definition}
We define the distance between two probability measures $\mu,\nu$ as
$\Delta(\mu,\nu)=\inf D(f,g)$, where infimum goes over all measurable
functions $f,g:I\rightarrow X$, satisfying $\mu=f_{\sharp}\lambda$,
$\nu=g_{\sharp}\lambda$.
\end{definition}

We will now prove in several steps that $P(X)$ equipped with $\Delta$ is
indeed a metric space. We will use the following form (as in e.g.
\cite{Denker:76}, Proposition 2.17) of the well known isomorphism Theorem
(\cite{Halmos:50}, Theorem C of Section 41.):

\begin{theorem}
If $X$ is a compact metric space with a nonatomic Borel probability measure
$\mu$, then it is isomorphic (in the category of measure spaces)\ to the the
Lebesgue measure $\lambda$ on the family of Borel measurable subsets of $I$.
\end{theorem}

\begin{corollary}
\label{t:big}Suppose $\mu$ is a Borel probability measure on a compact space
$X$. Then there exists a Borel measurable function $f:I\rightarrow X$ such
that $\mu=f_{\sharp}\lambda$.
\end{corollary}

This implies that the infimum in the definition of $\Delta$ goes over a
nonempty set, hence $\Delta$ is well defined.

In the rest of the paper, "measurable" will always mean "Borel-measurable".

\begin{lemma}
\label{l:topology} Suppose $\mu\in P(X)$, and let $g_{1},g_{2}:I\rightarrow X$
be measurable functions, such that $\mu=g_{1\sharp}\lambda=g_{2\sharp}\lambda
$. Then for each $\varepsilon>0$, there exist measurable, $\lambda$-invariant
functions $h_{1},h_{2}:I\rightarrow I$ such that
\begin{equation}
D(g_{1}\circ h_{1},g_{2}\circ h_{2})\leq\varepsilon\text{.}\label{r:prop3}%
\end{equation}

\end{lemma}

\begin{proof}
Let $C_{1}$,$C_{2},...,C_{k}$ be any pairwise disjoint, measurable cover of
the support of $\mu$, such that each $C_{i}$ has diameter less than
$\varepsilon$ (such a cover exists because of compactness of $X$). Without
loss of generality we also assume that for all $i$, $\mu(C_{i})>0$. We define
a set $Y_{\varepsilon}\subseteq I\times I$, and a Borel probability measure
$\nu_{\varepsilon}$ on $Y_{\varepsilon}$, with%
\begin{align*}
Y_{\varepsilon,i}  & =g_{1}^{-1}(C_{i})\times g_{2}^{-1}(C_{i}),\quad
Y_{\varepsilon}=%
{\textstyle\bigcup_{i=1}^{k}}
Y_{\varepsilon,i},\\
\nu_{\varepsilon}(A)  & =%
{\textstyle\sum_{i=1}^{k}}
\lambda^{2}(A\cap Y_{\varepsilon,i})/\mu(C_{i})\text{,}%
\end{align*}
where $\lambda^{2}$ is the Lebesgue measure on $I\times I$. Since $\lambda
^{2}(Y_{\varepsilon,i})=\lambda(g_{1}^{-1}(C_{i}))\cdot\lambda(g_{2}%
^{-1}(C_{i}))=\mu(C_{i})^{2}$, one can easily check that $\nu_{\varepsilon}$
is a probability measure. By definition, for any $a\in Y_{\varepsilon}$,
\begin{equation}
|g_{1}(\pi_{1}(a))-g_{2}(\pi_{2}(a))|\leq\varepsilon\text{,}\label{r:prop1}%
\end{equation}
where $\pi_{1},\pi_{2}:I^{2}\rightarrow I$ are coordinate projections. For any
measurable $A\subseteq I$, $\nu_{\varepsilon}(A\times I)=%
{\textstyle\sum_{i=1}^{k}}
\lambda(A\cap g_{1}^{-1}(C_{i}))\cdot\lambda(g_{2}^{-1}(C_{i}))/\mu(C_{i})=%
{\textstyle\sum_{i=1}^{k}}
\lambda(A\cap g_{1}^{-1}(C_{i}))=\lambda(A)$, and similarly $\nu_{\varepsilon
}(I\times A)=\lambda(A)$, hence%
\begin{equation}
\pi_{1\sharp}\nu_{\varepsilon}=\pi_{2\sharp}\nu_{\varepsilon}=\lambda
\text{.}\label{r:prop2}%
\end{equation}

By using Corollary \ref{t:big}, we find a measurable function $h^{\ast
}:I\rightarrow I^{2}$ such that $\nu_{\varepsilon}=h_{\sharp}^{\ast}\lambda$,
and $\nu_{\varepsilon}(I)\subseteq Y_{\varepsilon}$. Now (\ref{r:prop1}) and
(\ref{r:prop2}) imply that $h_{1}=\pi_{1}\circ h^{\ast}$, $h_{2}=\pi_{2}\circ
h^{\ast}$ are the required functions.
\end{proof}

\begin{proposition}
\label{p:topology} The function $\Delta$ is a metric on $P(X)$.
\end{proposition}

\begin{proof}
Since $D$ is symmetric, so is $\Delta$. The claim $\Delta(\mu,\mu)=0$ is trivial.

Now suppose that $\Delta(\mu,\nu)=0$ for $\mu\neq\nu$. Choose any $h\in C(X)$,
and $\varepsilon>0$. Because of compactness of $X$, $h$ is uniformly
continuous, and there exists $\delta>0$ such that $d(x,y)<\delta$ implies
$|h(x)-h(y)|<\varepsilon$ for all $x,y\in X$. Now we find $f,g:I\rightarrow X$
such that $\mu=f_{\sharp}\lambda$, $\nu=g_{_{\sharp}}\lambda$, and
$D(f,g)<\delta$. Now
\begin{align*}
\left\vert \int_{X}hd\mu-\int_{X}hd\nu\right\vert  & =\left\vert \int
_{I}h(f(t))dt-\int_{I}h(g(t))dt\right\vert \\
& \leq\int_{I}|h(f(t))-h(g(t))|dt\ \leq\varepsilon.
\end{align*}
Since $\varepsilon$ and $h$ were arbitrary, we see that $\mu$, $\nu$ are
identical linear functionals on $C(X)$. Now the Riesz Representation Theorem
implies that $\mu=\nu$.

Finally, we prove the triangle inequality. Let $\eta,\mu,\nu\in P(X)$. Choose
arbitrary $\varepsilon>0$, and assume that $f_{1},g_{1},g_{2},f_{2}%
:I\rightarrow X$ are measurable functions such that $\eta=f_{1\sharp}\lambda$,
$\mu=g_{1\sharp}\lambda=g_{2\sharp}\lambda$, $\nu=f_{2\sharp}\lambda$, and
such that
\begin{equation}
\Delta(\eta,\mu)\geq D(f_{1},g_{1})-\varepsilon,\quad\Delta(\mu,\nu)\geq
D(g_{2},f_{2})-\varepsilon\text{.}\label{r:getit1}%
\end{equation}

Now we find $h_{1},h_{2}$ as in Lemma \ref{l:topology}. Note now that for
arbitrary functions $f^{\ast},g^{\ast}:I\rightarrow X$, $h^{\ast}:I\rightarrow
I$, $D(f^{\ast},g^{\ast})\geq D(f^{\ast}\circ h^{\ast},g^{\ast}\circ h^{\ast
})$, so first applying\ that, then (\ref{r:prop3}), and finally the triangle
inequality for $D$ we get:%
\begin{align}
D(f_{1},g_{1})+D(g_{2},f_{2})  & \geq D(f_{1}\circ h_{1},g_{1}\circ
h_{1})+D(g_{2}\circ h_{2},f_{2}\circ h_{2})\geq\nonumber\\
& \geq D(f_{1}\circ h_{1},g_{1}\circ h_{1})+(D(g_{1}\circ h_{1},g_{2}\circ
h_{2})-\varepsilon)+\nonumber\\
& +D(g_{2}\circ h_{2},f_{2}\circ h_{2})\nonumber\\
& \geq D(f_{1}\circ h_{1},f_{2}\circ h_{2})-\varepsilon\text{.}%
\label{r:getit4}%
\end{align}
Since $h_{1},h_{2}$ are $\lambda$-invariant, $(f_{1}\circ h_{1})_{\sharp
}\lambda=\eta$, $(f_{2}\circ h_{2})_{\sharp}\lambda=\nu$, hence
\begin{equation}
D(f_{1}\circ h_{1},f_{2}\circ h_{2})\geq\Delta(\eta,\nu)\label{r:getit3}%
\end{equation}

Combining (\ref{r:getit1}), (\ref{r:getit4}) and (\ref{r:getit3}) we conclude
that $\Delta(\eta,\mu)+\Delta(\mu,\nu)\geq\Delta(\eta,\nu)-3\varepsilon$.
Since $\varepsilon$ was arbitrary, $\Delta$ satisfies the triangle inequality.
\end{proof}

\begin{definition}
The topology on $P(X)$ induced by the metric $\Delta$ is called dynamical
topology. We denote it by $\tau_{d}$.
\end{definition}

\section{Properties of the dynamical topology}

We now compare different topologies on $P(X)$, and investigate elementary
properties of the dynamical topology.

\begin{proposition}
\label{p:sub} The weak* topology on $P(X)$ is coarser than $\tau_{d}$.
Equivalently, d-convergence implies w-convergence.
\end{proposition}

\begin{proof}
Assume that a sequence $\mu_{n}\in P(X)$ d-converges to a $\mu\in P(X)$.
Choose an arbitrary $f\in C(X)$, and $\varepsilon>0$. Since $X$ is compact,
$f$ is uniformly continuous, and there is $\delta>0$ such that $|x-y|<\delta$
implies $|f(x)-f(y)|<\varepsilon$. Now choose $n_{0}$ large enough such that
$\Delta(\mu_{n},\mu)<\delta/2$ for $n>n_{0}$, and for given $n$ choose
$g_{n},g$ such that $\mu_{n}=g_{n\sharp}\lambda$, $\mu=g_{\sharp}\lambda$, and
$D(g_{n},g)<\delta$. Now
\begin{align*}
|\mu_{n}(f)-\mu(f)|  & =\left\vert \int_{X}fd\mu_{n}-\int_{X}fd\mu\right\vert
=\left\vert \int_{I}(f\circ g_{n}-f\circ g)d\lambda\right\vert \leq\\
& \leq\int_{I}\left\vert f\circ g_{n}-f\circ g\right\vert d\lambda\leq\int
_{I}\varepsilon d\lambda\leq\varepsilon,
\end{align*}
therefore $\mu_{n}$ w-converges to $\mu$.
\end{proof}

The next simple example shows that the dynamical topology differs from both
the weak* and uniform topology on any nontrivial $X$ (i.e. $X$ with more than
one element). We will see that the dynamical topology refines the weak*
topology in a very different way than the uniform topology.

\begin{example}
Suppose that $\mu_{n}$ is a sequence of atomic measures, each supported on $k$
points $x_{i}^{n}\in X$, $i=1,...,k$, $n\in N$. We write%
\[
\mu_{n}=\sum_{i=1,...,k}p_{i}^{n}\delta(x_{i}^{n}),
\]
where $x_{i}^{n}\in X$, $p_{i}^{n}\geq0$, and $\sum_{i=1}^{k}p_{i}^{n}=1$.
Suppose now that $\mu_{n}$ $w$-converges to $\mu=\sum_{i=1,...,k}q_{i}%
\delta(y_{i})$. Without loss of generality we can assume that $p_{i}%
^{n}\rightarrow q_{i}$, and $x_{i}^{n}\rightarrow y_{i}$. It is easy to check
that $\mu_{n}$ is $u$-convergent if and only if for all $i$, $x_{i}^{n}$ is
eventually constant (i.e. there exists $n_{0}$ such that for all $n\geq n_{0}%
$, $x_{i}^{n}=y_{i}$).

On the other hand, one can check that the sequence$\,(\mu_{n})$ is
$d$-convergent, if and only if for all $i$, $p_{i}^{n}$ is eventually constant.

We also deduce that in this example $(\mu_{n})$ is at the same time $u$- and
$d$-convergent, if and only if it is eventually constant.
\end{example}

Proposition \ref{p:sub} and the Example above imply the following conclusion.

\begin{corollary}
\label{c:horprop} (i)\ If $X$ has at least two elements, then $\tau_{w}%
\subset\tau_{d}$, but not equal to it;

(ii) If $X$ is not a finite set, then $\tau_{d}\not \subset \tau_{u}$ and
$\tau_{u}\not \subset \tau_{d}$;
\end{corollary}

Now we discuss w-connectedness and w-compactness of $P(X)$.

\begin{proposition}
\label{p:properties} (i) If $X$ is path-connected, then $P(X)$ is $\,$d-path connected.

(ii) If $X$ has at least two elements, than $P(X)$ with the dynamical topology
is not d-sequentially compact, and not d-compact.
\end{proposition}

\begin{proof}
(i) Suppose that $X$ is path-connected. Let $\mu$, $\mu^{\prime}$ be any two
measures in $P(X)$, and choose any measurable $f,f^{\prime}:I\rightarrow X$
such that $\mu=f_{\sharp}I$, $\mu^{\prime}=f_{\sharp}^{\prime}I$. Now, since
$X$ is path connected, there is a measurable function $g:I\times I\rightarrow
X$, such that $g(.,0)=f$, $g(.,1)=f^{\prime}$, and such that $t\mapsto g(a,t)$
is continuous for every $a$. Now the function $t\mapsto g(.,t)_{\sharp}%
\lambda$ is a d-continuous curve in $P(X) $, connecting $\mu$ and $\mu
^{\prime}$.

(ii) We construct the following simple example: choose two points $x\not =y$
in $X$, and the sequence of measures
\begin{equation}
\mu_{n}=\frac{1}{n}\delta_{x}+\frac{n-1}{n}\delta_{y},\label{examp}%
\end{equation}
where $\delta_{x}$, $\delta_{y}$ are atomic measures concentrated in $x$, $y$.
Now $\mu_{n}$ $w$-converges to $\mu=\delta_{y}$. Since $\Delta(\mu_{n}%
,\mu)=d(x,y)$, neither $\mu_{n}$ nor any subsequence of $\mu_{n}$ d-converge
to $\mu$. Proposition \ref{p:sub} implies that $\mu_{n}$ has no convergent
subsequence, hence $P(X)$ is not d-sequentially compact. Since $P(X)$ is
metrizable, $P(X)$ is not d-compact.
\end{proof}

We now develop several simple tools used in proofs later in the paper. For a
given set $J\subseteq I$ and measurable $f,g:I\rightarrow X$, we define
$D_{J}(f,g)=\sup_{a\in J}d(f(a),g(a))$ \ We denote the support of a measure
$\mu$ by supp$(\mu)$ .

\begin{lemma}
\label{l:topA}Suppose $J\subseteq I$ is a measurable set of full measure.

(i)$\ $For any measurable $f,g:I\rightarrow X,$ there exist measurable
$\widetilde{f},\widetilde{g}:I\rightarrow X$, such that $f_{\sharp}%
\lambda=\widetilde{f}_{\sharp}\lambda$, $g_{\sharp}\lambda=\widetilde
{g}_{\sharp}\lambda$, and $D(\widetilde{f},\widetilde{g})\leq D_{J}(f,g)$.

(ii) The distance $\Delta(\mu,\nu)=\inf D_{J}(f,g)$, where infimum goes over
all measurable functions $f,g:I\rightarrow X$, satisfying $\mu=f_{\sharp
}\lambda$, $\nu=g_{\sharp}\lambda$.
\end{lemma}

\begin{proof}
(i)\ Choose any $x\in X$, and define $\widetilde{f}(a)=f(a)$, $\widetilde
{g}(a)=g(a)$ for $a\in J$, $\widetilde{f}(a)=\widetilde{g}(a)=x$ for
$a\not \in J $. (ii) It follows from $D(\widetilde{f},\widetilde{g})\leq
D_{J}(f,g)\leq D(f,g)$.
\end{proof}

\begin{lemma}
\label{l:simple}If $f,g:I\rightarrow X$ are measurable functions such that
$\mu=f_{\sharp}\lambda$, $\nu=g_{\sharp}\lambda$ for given $\mu,\nu\in P(X)$,
then there exist measurable functions $\widetilde{f},\widetilde{g}%
:I\rightarrow X,$ such that $\mu=\widetilde{f}_{\sharp}\lambda$,
$\nu=\widetilde{g}_{\sharp}\lambda$, $\widetilde{f}(I)\subseteq$supp$(\mu)$,
$\widetilde{g}(I)\subseteq$supp$(\nu)$, and such that $D(\widetilde
{f},\widetilde{g})\leq D(f,g)$.
\end{lemma}

\begin{proof}
Let $J=f^{-1}($supp$(\mu))\cap g^{-1}($supp$(\nu))$, and choose any $t_{0}\in
J$. Now we define $\widetilde{f}(t)=f(t)$, $\widetilde{g}(t)=g(t)$ for $t\in
J$; $\widetilde{f}(t)=f(t_{0})$, $\widetilde{g}(t)=g(t_{0})$ for $t\not \in J$.
\end{proof}

Among various (uniformly equivalent but not necessarily
equivalent)\ definitions of the Hausdorff metric $d_{H}$ in the literature, we
use the following as the most convenient here:\ if $A$, $B$ are two closed
subsets of $X$, then $d(x,A)=\inf\{d(x,y),y\in A\}$, $d(A,B)=\sup
\{d(x,B)\,|\,x\in A\}$, and $d_{H}(A,B)=\max\{d(A,B),d(B,A)\}$.

\begin{proposition}
\label{p:haus} (i) $d_{H}($supp$(\mu),$supp$(\nu))\leq\Delta(\mu,\nu)$;

(ii)\ If $\mu_{n}$ d-converges to $\mu$, then supp$(\mu_{n})$ converges to
supp$(\mu)$ in the Hausdorff topology.
\end{proposition}

\begin{proof}
(i)\ Suppose $\mu$, $\nu\in P(X)$. For any $\varepsilon>0$ there exist
measurable functions $f,g:I\rightarrow X$, such that $\mu=f_{\sharp}\lambda$,
$\nu=g_{\sharp}\lambda$, and such that $D(f,g)<\Delta(\mu,\nu)+\varepsilon$.
Using Lemma \ref{l:simple}, we construct $\widetilde{f} $, $\widetilde{g}$, as
in the Lemma. We easily check that $d_{H}($supp$(\mu),$supp$(\nu))\leq
D(\widetilde{f},\widetilde{g})\leq D(f,g)<\Delta(\mu,\nu)+\varepsilon$. Since
$\varepsilon$ was arbitrary, (i) is proved. (ii)\ follows directly from (i).
\end{proof}

The claim (ii) of the Proposition \ref{p:haus} is not true in the weak*
topology. The counter-example which is constructed in (\ref{examp}) is a
sequence of w-convergent measures $\mu_{n},$ converging to a measure $\mu$,
but such that supports of measures $\mu_{n}$ do not converge to the support of
the measure $\mu$.

Assume that a sequence of measures $\mu_{n}$ w-converges to a measure $\mu$,
and that the sequence of supports of measures $\mu_{n}$ converges in the
Hausdorff topology to the support of $\mu$. We can not then in general claim
that $\mu_{n}$ d-converges to $\mu$. A counter-example is the sequence
\[
\mu_{n}=\frac{n+1}{2n}\delta_{x}+\frac{n-1}{2n}\delta_{y}\text{,}%
\]
for $\delta_{x}$, $\delta_{y}$ as in (\ref{examp}). The same example shows
that we can find measures $\mu$, $\nu$, such that $d_{H}($supp$(\mu
),$supp$(\nu))=0$, but such that $\Delta(\mu,\nu)$ is arbitrarily large.

We now show that the dynamical topology is \textit{not} much finer than the
weak* topology.

\begin{proposition}
\label{p:dense}The set of all measures in $P(X)$ which are supported on a
finite set is d-dense.
\end{proposition}

\begin{proof}
Choose $\mu\in P(X)$, and any $\varepsilon>0$. Since $X$ is compact, we can
find a measurable pairwise disjoint cover $C_{1},C_{2},...,C_{k}$ of $X$, such
that the diameter of each $C_{i}$ is at most $\varepsilon$. For each $1\leq
i\leq k$, choose any $x_{i}\in C_{i}$. Let $\nu_{\varepsilon}=%
{\textstyle\sum_{i=1}^{k}}
\mu(C_{i})\delta_{x_{i}}$, and then $\nu_{\varepsilon}$ is supported on the
finite set $\left\{  x_{1},...,x_{k}\right\}  $. Choose any measurable
$f:I\rightarrow X$ such that $\mu=f_{\sharp}\lambda$, and define
$g:I\rightarrow X$ \ with $g(t)=x_{i}$ for all $t\in f^{-1}(C_{i})$. Now,
$\nu_{\varepsilon}=g_{\sharp}\lambda$, and since diameter of each $C_{i}$ is
at most $\varepsilon$, $D(f,g)\leq\varepsilon$, hence $\Delta(\mu
,\nu_{\varepsilon})\leq\varepsilon$.
\end{proof}

With regards to the weak* topology, the set of all measures uniformly
supported on a finite (multi)set is w-dense. No similar claim is true in the
uniform topology.

\section{Characterization of convergence in the dynamical topology}

We first recall some well known properties of convergent sequences of measures.

\begin{proposition}
\label{p:basic} Assume that $\mu_{n}$ $w$-converges, $d$-converges or
$u$-converges to $\mu$. Then

(i) For each open $U\subset X$, $\liminf\mu_{n}(U)\geq\mu(U)$.

(ii) For each closed $V\subset X$, $\limsup\mu_{n}(V)\leq\mu(V)$.

(iii) For each $W\subset X$ such that $\mu(\partial(W))=0$, then $\lim\mu
_{n}(W)=\mu(W)$.
\end{proposition}

\begin{proof}
The proof for $w$-convergence is e.g. in \cite{Walters:82}. The rest follows
from Proposition \ref{p:sub} and the definition of u-convergence.
\end{proof}

The following notion is the main tool for characterization of d-convergence.

\begin{definition}
We say that a measurable set $A\subset X$ is $\mu$-separating, if there exists
an open set $D$, $Cl(A)\subseteq D$, such that $\mu(D\setminus A)=0$.
\end{definition}

Note that a $\mu$-separating $A$ can have measure $0$.

In the following, $\epsilon$-neighborhood of a set $A$ is the open set $\{x\in
X,$ such that $\exists y\in A,d(x,y)<\epsilon\}$. Note also that,\ since $X$
is compact, if $D$ is any open set such that $Cl(A)\subseteq D$, then for
small $\varepsilon$, an $\varepsilon$-neighborhood of $A$ is a subset of $D$.

The proof of the following Lemma is an easy exercise.

\begin{lemma}
\label{p:sep} Assume that $A\subset X$ is a measurable set. $A$ is $\mu
$-separating, if and only if for each small enough $\epsilon>0$, if $D$ is an
$\epsilon$-neighborhood of $A$, then $\mu(D\setminus A)=0$.
\end{lemma}

Now we can characterize d-convergence.

\begin{theorem}
\label{t:char} A sequence of measures $\mu_{n}$ d-converges to a measure $\mu$
if and only if the following two conditions hold:

(i) The sequence $\mu_{n}$ w-converges to the measure $\mu$; and

(ii) For each $\mu$-separating set $A$, there is a neighborhood $B$ of $A$,
such that for any open $C$, $Cl(A)\subseteq C\subseteq B$, there is $n_{0}$
such that for all $n>n_{0}$, $\mu_{n}(C)=\mu(C)$.
\end{theorem}

We prove Theorem \ref{t:char} in several steps.

\begin{lemma}
\label{l:claim1}\textit{Suppose }$\mu_{n}$, $\mu$ satisfy conditions (i),
(ii)\ of Theorem \ref{t:char}, and\textit{\ choose any }$\epsilon>0$\textit{.
Assume that }$A_{1}$\textit{, }$A_{2}$\textit{, ..., }$A_{m}$\textit{\ is a
cover of }$X$, of \textit{measurable, nonempty, pairwise disjoint sets with
diameter less or equal than }$\epsilon$\textit{. Then we can find }$\delta$,
$0<\delta\leq\epsilon$\textit{, and an integer }$n_{0}$, \textit{such that, if
}$B_{1}$\textit{, ..., }$B_{m}$\textit{\ are }$\delta$\textit{-neighborhoods
of }$A_{1}$\textit{,..., }$A_{m}$\textit{\ respectively, then for all }$n\geq
n_{0}$\textit{, and any }$1\leq i_{1}<i_{2}<...<i_{k}\leq m$\textit{,}%
\begin{equation}
\mu_{n}(B_{i_{1}}\cup B_{i_{2}}\cup...\cup B_{i_{k}})\geq\mu(A_{i_{1}}\cup
A_{i_{2}}\cup...\cup A_{i_{k}}).\label{e:smart}%
\end{equation}

\end{lemma}

\begin{proof}
Let $\mathcal{I}$ be the set of all subsets of indices $\{i_{1},i_{2}%
,...,i_{k}\}\subseteq\{1,2,...,m\}$, for which $A_{i_{1}}\cup A_{i_{2}}%
\cup...\cup A_{i_{k}}$ is $\mu$-separating. Now choose $\delta\leq\epsilon$
small enough, such that the condition (ii) of the Theorem is satisfied in the
following sense: for all $(i_{1},i_{2},...,i_{k})\in\mathcal{I}$, if $B_{1}%
$,..., $B_{m}$ are $\delta$-neighborhoods of $A_{1}$,..., $A_{m}$
respectively, then there is $n_{1}$ such that for all $n\geq n_{1}$,
\begin{equation}
\mu(B_{i_{1}}\cup B_{i_{2}}\cup...\cup B_{i_{k}})=\mu_{n}(B_{i_{1}}\cup
B_{i_{2}}\cup...\cup B_{i_{k}}).\label{e:00}%
\end{equation}
For all $(i_{1},i_{2},...,i_{k})\in\mathcal{I}$, since $(A_{i_{1}}\cup
A_{i_{2}}\cup...\cup A_{i_{k}})$ is $\mu$-separating, we can also assume that
$\delta$ is small enough, so that
\begin{equation}
\mu(B_{i_{1}}\cup B_{i_{2}}\cup...\cup B_{i_{k}})=\mu(A_{i_{1}}\cup A_{i_{2}%
}\cup...\cup A_{i_{k}}).\label{e:01}%
\end{equation}

The converse of Lemma \ref{p:sep} implies that if $(i_{1},i_{2},...,i_{k}%
)\notin\mathcal{I}$,
\begin{equation}
\mu(B_{i_{1}}\cup B_{i_{2}}\cup...\cup B_{i_{k}})>\mu(A_{i_{1}}\cup A_{i_{2}%
}\cup...\cup A_{i_{k}}).\label{e:02}%
\end{equation}
Now, we define
\begin{equation}
\rho=\min\{\mu(B_{i_{1}}\cup B_{i_{2}}\cup...\cup B_{i_{k}})-\mu(A_{i_{1}}\cup
A_{i_{2}}\cup...\cup A_{i_{k}}),(i_{1},i_{2},...,i_{k})\notin\mathcal{I}%
\}.\label{e:03}%
\end{equation}
Then because of (\ref{e:02}), $\rho>0$. Now, choose $n_{2}$ such that, for all
$n\geq n_{2}$, and for all $(i_{1},i_{2},...,i_{k})\notin\mathcal{I}$
\begin{equation}
\mu_{n}(B_{i_{1}}\cup B_{i_{2}}\cup...\cup B_{i_{k}})\geq\mu(B_{i_{1}}\cup
B_{i_{2}}\cup...\cup B_{i_{k}})-\rho.\label{e:04}%
\end{equation}
Such $n_{2}$ exists because $B_{i_{1}}\cup B_{i_{2}}\cup...\cup B_{i_{k}}$ is
open, because $\mu_{n}$ w-converges to $\mu$, and because of Proposition
\ref{p:basic}, (i). Let $n_{0}=\max\{n_{1},n_{2}\}$. Now, (\ref{e:00}) and
(\ref{e:01}) imply (\ref{e:smart}) for $(i_{1},i_{2},...,i_{k})\in\mathcal{I}%
$, and (\ref{e:03}) and (\ref{e:04}) imply (\ref{e:smart}) for $(i_{1}%
,i_{2},...,i_{k})\notin\mathcal{I}$.
\end{proof}

\begin{lemma}
\label{l:claim2}Let $\xi$\ be a measure on $X$\ (positive, not necessarily a
normed one), $x_{1},x_{2},...,x_{m}$, nonnegative real numbers, and
$B_{1},B_{2},...,B_{m}$\ measurable subsets of $X$, such that for any $1\leq
i_{1}<i_{2}<...<i_{k}\leq m$,\textit{\ }%
\begin{align}
\xi(B_{i_{1}}\cup B_{i_{2}}\cup...\cup B_{i_{k}})  & \geq x_{i_{1}}+x_{i_{2}%
}+...+x_{i_{k}},\label{com1}\\
\xi(X)  & =x_{1}+x_{2}+...+x_{m}.\label{com2}%
\end{align}
\textit{Then there exist measures }$\nu_{1},\nu_{2},...,\nu_{m}$%
\textit{\ (positive, not necessarily normed) such that for all }%
$i=1,...,m$\textit{,}%
\begin{align}
\nu_{i}(B_{i}^{C})  & =0\text{,}\label{com3.2}\\
\nu_{i}(X)  & =x_{i}\text{,}\label{com3.3}\\
\xi & =\nu_{1}+\nu_{2}+...+\nu_{m}.\label{com3.4}%
\end{align}

\end{lemma}

The proof of Lemma \ref{l:claim2} is essentially combinatorial in character
and is not related to the rest of the paper, so we postpone its proof to the Appendix.

\begin{lemma}
\label{l:claim4}Suppose $\mu_{n}$, $\mu$ satisfy conditions (i), (ii)\ of
Theorem \ref{t:char}, and choose any $\epsilon>0$. Then there is an integer
$n_{0}$ such that for all $n\geq n_{0}$,%
\[
\Delta(\mu_{n},\mu)\leq2\epsilon.
\]

\end{lemma}

\begin{proof}
We can find a finite cover $A_{1}$, $A_{2}$, ..., $A_{m}$\ of $X$\ of
measurable, nonempty, pairwise disjoint sets with diameter less or equal than
$\varepsilon$ because of compactness of $X$. We apply Lemma \ref{l:claim1} and
find $0<\delta\leq\epsilon$ and an integer $n_{0}$, such that if $B_{1}%
,B_{2},...,B_{m}$ are the $\delta$-neighborhoods of $A_{1}$, $A_{2}$, ...,
$A_{m}$, and $n\geq n_{0}$, then (\ref{e:smart}) holds. Choose $n\geq n_{0}$,
and set $x_{i}=\mu(A_{i})$, $\xi=\mu_{n}$. The relation (\ref{e:smart}) and
the fact that $(A_{i})$ are pairwise disjoint imply (\ref{com1}). As $(A_{i})$
is a measurable partition of $X$, $\mu(A_{1})+...+\mu(A_{m})=\mu(X)=1=\mu
_{n}(X)$, which is by definition (\ref{com2}). Now applying Lemma
\ref{l:claim2} we obtain positive measures $\nu_{1},\nu_{2},...,\nu_{m}$ such
that
\begin{align*}
\mu_{n}  & =\nu_{1}+\nu_{2}+...+\nu_{m},\\
\nu_{i}(B_{i}^{c})  & =0\text{,}\\
\nu_{i}(X)  & =\nu_{i}(B_{i})=\mu(A_{i})\text{.}%
\end{align*}
Let $I_{1}=[a_{0},a_{1})$, $I_{2}=[a_{1},a_{2})$, ... ,$I_{m}=[a_{m-1},a_{m}]$
be a partition of $[0,1]$, $0=a_{0}\leq a_{1}\leq a_{2}\leq...\leq a_{m-1}\leq
a_{m}=1$, such that $a_{i}-a_{i-1}=\mu(A_{i})$ for all $i=1,...,m$ (if
$a_{i-1}=a_{i}$, then $I_{i}=\emptyset$). Then for all $i=1,...,m$,
\[
\nu_{i}(B_{i})=\mu(A_{i})=\lambda(I_{i}).
\]
Now, using Theorem \ref{t:big}, we construct a function $f:I\rightarrow X$
such that $f(I_{i})\subseteq A_{i}$, $i=1,...,m$, and such that $\mu
=f_{\sharp}\lambda$. The construction relies on the fact that $\{A_{i}%
,i=1,...,m\}$ is a measurable partition of $X$, and $\mu(A_{i})=\lambda
(I_{i})$.

Similarly, we can construct a function $f_{n}:I\rightarrow X$ such that
$f_{n}(I_{i})\subseteq B_{i}$, and such that $f_{n\sharp}\lambda=\mu_{n}$. We
do it in the following way: let $\lambda_{i}=\lambda\,|\,I_{i}$. We construct
$f_{n}|_{I_{i}}$ to be any measurable function so that
\[
(f_{n}|_{I_{i}})_{\sharp}\lambda_{i}=\nu_{i}.
\]

The construction implies that $f(t)\in A_{i}$ if and only if $f_{n}(t)\in
B_{i}$. Since $A_{i}\subseteq B_{i}$, and the diameter of $B_{i}$ is at most
$2\epsilon$, then for any $x\in A_{i}$, $y\in B_{i}$ we get $d(x,y)\leq
2\varepsilon$. We conclude that $D(f,f_{n})\leq2\epsilon$, and by definition
$\Delta(\mu,\mu_{n})\leq2\epsilon$.
\end{proof}

\begin{lemma}
\label{l:claim5}If a sequence of measures $\mu_{n}$ d-converges to a measure
$\mu$, then it satisfies (ii)\ from Theorem \ref{t:char}.
\end{lemma}

\begin{proof}
Assume that $A$ is $\mu$-separating, and let $D$ be an open set, such that
$Cl(A)\subseteq D$, and such that $\mu(D\setminus A)=0$. Let $\varepsilon>0$
such that $2\varepsilon$-neighborhood of $A$ is a subset of $D$. Let $B$ be
the $\varepsilon$-neighborhood of $A$, and choose an arbitrary open set $C$
such that $Cl(A)\subseteq C\subseteq B$. Then the construction implies that
\begin{equation}
x\in C,\,y\in D^{c}\Longrightarrow d(x,y)>\varepsilon.\label{e:50}%
\end{equation}

Choose $\delta<\varepsilon$, small enough such that the $\delta$ neighborhood
$A$ is a subset of $C$. Now find $n_{0}$ large enough such that for all $n\geq
n_{0}$, $\Delta(\mu,\mu_{n})<\delta/2$. Now we can find functions $f$, $f_{n}$
such that
\begin{equation}
D(f,f_{n})<\delta,\label{e:51}%
\end{equation}
and such that $\mu=f_{\sharp}\lambda,\mu_{n}=f_{n\sharp}\lambda$. Definitions
of $A$ and $D$ imply that $\mu(D\setminus A)=0$. Now Lemma \ref{l:topA}, (i)
implies that without loss of generality we can assume that for all
\begin{equation}
x\in I,\,f(x)\notin D\setminus A.\label{e:52}%
\end{equation}

Suppose $f(x)\in A$. Then (\ref{e:51}) implies that $d(f(x),f_{n}(x))<\delta$,
so $f_{n}(x)$ is in $\delta$ neighborhood of $A$, which is a subset of $C$.
Now, suppose $f_{n}(x)\in C$. Now, because of (\ref{e:50}) and (\ref{e:51}),
\thinspace$f(x)\not \in D^{c}$, and because of (\ref{e:52}), $f(x)\not \in
D\setminus A$, so it must be $f(x)\in A$. We deduce that $f_{n}(x)\in C$\ if
and only if $f(x)\in A$, hence $\mu_{n}(C)=\lambda(f_{n}^{-1}(C))=\lambda
(f^{-1}(A))=\mu(A)$. Since $A$ is $\mu$-separating, $\mu(A)=\mu(C)$, so we
deduce that for all $n\geq n_{0}$, $\mu_{n}(C)=\mu(C)$.
\end{proof}

We now prove Theorem \ref{t:char}.

\begin{proof}
One implication of Theorem \ref{t:char} follows from Lemma \ref{l:claim4}; the
other from Proposition \ref{p:sub} and Lemma \ref{l:claim5}.
\end{proof}

Theorem \ref{t:char} could be further modified: to check whether a
w-convergent sequence is d-convergent,\ it is sufficient to prove (ii)\ only
for closed $\mu$-separating sets.

The following important Corollary shows that the dynamical topology is in some
sense sufficiently close and similar to the weak*\ topology, and as such is
expected to have various applications.

\begin{corollary}
Assume that $X$ is connected. If supp $\mu=X$, then $\mu_{n}$ d-converges to
$\mu$ if and only if it w-converges to $\mu$.
\end{corollary}

\begin{proof}
Indeed, if $X$ is connected and supp $\mu=X$, then there are no $\mu
$-separating sets.
\end{proof}

Now we give yet another characterization of d-convergence. By definition, a
sequence $\mu_{n}\in P(X)$ d-converges to a measure $\mu\in P(X)$, if there
exists a sequence of measurable functions $f_{n},g_{n}:I\rightarrow X$, such
that $\mu_{n}=(f_{n})_{\sharp}\lambda$, $\mu=(g_{n})_{\sharp}\lambda$, and
$D(f_{n},g_{n})\rightarrow0$. Now we show that $g_{n}$ can be independent of
$n$, and put it in the context of the well-known Skorokhod Theorem (see e.g.
\cite{Jacobs:78}).

\begin{theorem}
(Skorokhod) Assume that $X$ is metrizable, separable, and complete. A sequence
of measures $\mu_{n}\in P(X)$ is $w$-convergent, if and only if there exists a
sequence of measurable functions $f_{n}:I\rightarrow X$, and a function
$f:I\rightarrow X$, such that for all $a\in I$, $f_{n}(a)\rightarrow f(a)$,
and such that $\mu_{n}={f_{n}}_{\sharp}\lambda$, and $\mu=f_{\sharp}\lambda$.
\end{theorem}

\begin{corollary}
Assume that $X$ is metrizable. A sequence of measures $\mu_{n}\in P(X)$ is
d-convergent, if and only if there exists a sequence of measurable functions
$f_{n}:I\rightarrow X$, and a function $f:I\rightarrow X$, such that
$f_{n}\rightarrow f$ as $n\rightarrow\infty$, \textbf{uniformly} on $I$, and
such that $\mu_{n}={f_{n\sharp}}\lambda$, and $\mu=f_{\sharp}\lambda$.
\end{corollary}

\begin{proof}
$\Longleftarrow:$ It follows directly from the definition of $\Delta$.
\medskip

$\Longrightarrow:$ Assume that $\mu_{n}$ d-converges to $\mu$. We can adjust
the construction of the function $f:I\rightarrow X$ in the proof of the
Theorem \ref{t:char} so that $\mu=f_{\sharp}\lambda$, and so that $f$ is
independent of $\varepsilon$. We do it by choosing $\varepsilon=1/2^{n}$, and
constructing the cover $A_{1},...,A_{m}$ for a chosen $\varepsilon$ to be a
refinement of the cover $A_{1}^{\prime},...,A_{m^{\prime}}^{\prime}$ for
another $\varepsilon^{\prime}>\varepsilon$. We can also, using the same
construction, find a sequence $f_{n}:I\rightarrow X$ so that $D(f_{n}%
,f)\rightarrow0$, and that $\mu_{n}=f_{n\sharp}\lambda$, which proves the claim.
\end{proof}

\section{Stability of invariant measures}

We now show that various notions of stability, including Lyapunov stability,
attracting sets, asymptotic stability, and Nekhoroshev stability, generalize
well to invariant measures. As in the previous sections, $X$ is a compact
metric space. In this section, $f$ is a continuous function on $X$, and
$f_{\sharp}$ is then a d-continuous function on $P(X)$, $f_{\sharp}\mu
(A)=\mu(f^{-1}(A))$ for all\ Borel measurable $A$.

\begin{definition}
For all Borel measurable $A\subseteq X$, we call the set $A_{\sharp}$ the
\textbf{lift} of the set $A$, defined as the set of all measures $\mu\in
P(X)$, such that supp$(\mu)\subseteq A$.
\end{definition}

\begin{lemma}
\label{l:elementary}Let $A,B,A_{i}$, $i=1,...,\infty$, be measurable subsets
of $X$. Then (i)\ $A=B$ if and only if $A_{\sharp}=B_{\sharp}$%
;\ (ii)\ $A\subseteq B$ if and only if $A_{\sharp}\subseteq B_{\sharp}$; (iii)
$(A\cap B)_{\sharp}=A_{\sharp}\cap B_{\sharp}$; (iv) $(%
{\textstyle\bigcap_{i=1}^{\infty}}
A_{i})_{\sharp}=$ $%
{\textstyle\bigcap_{i=1}^{\infty}}
(A_{i})_{\sharp}$; (v) $(A^{c})_{\sharp}\subseteq(A_{\sharp})^{c}$; (vi)
$A_{\sharp}\cup B_{\sharp}\subseteq(A\cup B)_{\sharp}$.
\end{lemma}

All properties (i)-(vi)\ follow directly from the definition of the lift. Note
that in general in (v) and (vi)\ equality does not hold, so lift $\sharp$ is
not a morphism of set algebras.

\begin{lemma}
\label{l:seemstrivial}If $f:X\rightarrow X$ is continuous, then $f_{\sharp
}(A_{\sharp})=f(A)_{\sharp}$.
\end{lemma}

\begin{proof}
$\subseteq$: Let $\mu\in f_{\sharp}(A_{\sharp})$. Then for some $\nu\in P(X)$,
$\mu=f_{\sharp}\nu$, and supp$(\nu)\subseteq A$. Since $f$ is continuous and
$X$ compact, $f($supp $(\nu))=$supp$(f_{\sharp}\nu)=$supp$(\mu)$, hence
supp$(\mu)\subseteq f(A)$.

$\supseteq:$ Let $\mu\in f(A)_{\sharp}$, i.e. supp$(\mu)\subseteq f(A)$. Let
$\mu^{n}=\sum_{k=1}^{m_{n}}\lambda_{k}^{n}\delta(y_{k}^{n})$ be any sequence
of finitely supported measures which w-converges to $\mu$, and such that for
all $n,k$, $y_{k}^{n}\in$supp$(\mu)$; let $x_{k}^{n}\in A$ be any sequence
such that $f(x_{k}^{n})=y_{k}^{n}$, and let $\nu$ be the limit point of a
w-convergent subsequence of $\nu^{n}=\sum_{k=1}^{m_{n}}\lambda_{k}^{n}%
\delta(x_{k}^{n})$. Then supp$(\nu)\subseteq f^{-1}($supp$(\mu))\subseteq A$,
and because of continuity of $f$, $\mu=f(\nu)$, hence $\mu\in f_{\sharp
}(A_{\sharp})$.
\end{proof}

In the following, $\varepsilon$-neighborhoods of sets of measures and other
properties in $P(X)$ are always with respect to the dynamical topology, unless
specified otherwise.

The key property of the dynamical topology is the following lemma:

\begin{lemma}
\label{l:help}Suppose $A\subseteq X$ is a closed set, and let $U\subseteq X$
be an open set. Then $U$ is the $\varepsilon$-neighborhood of $A$ if and only
if $U_{\sharp}$ is the $\varepsilon$-neighborhood of $A_{\sharp}$.
\end{lemma}

\begin{proof}
$\Rightarrow:$ Denote by $\mathcal{V}$ the $\varepsilon$-neighborhood of
$A_{\sharp}$ in $P(X)$.

(i)\ \textit{Claim:\ }$U_{\sharp}\subseteq\mathcal{V}$\textit{.} Choose any
$\mu\in U_{\sharp}$, and then by definition supp$(\mu)\subseteq U$. Since $X$
is compact and supp$(\mu)$ closed, there exists $\delta>0$ such that
supp$(\mu)$ is a subset of a ($\varepsilon-2\delta$)-neighborhood of $A$. Now,
Proposition \ref{p:dense} implies that there is a measure $\nu_{\delta}%
\,$\ supported on a finite set $\left\{  x_{1},...,x_{k}\right\}  $ such that
$\Delta(\mu,\nu_{\delta})<\delta$. Without loss of generality we also assume
that $\nu_{\delta}(\left\{  x_{i}\right\}  )>0$ for all $i$, and then it is
easy to see that for all $1\leq i\leq k$, $x_{i}\,\ $is in the ($\varepsilon
-\delta$)-neighborhood of $A$. We now choose $y_{1},...,y_{k}\in A$, such that
$d(x_{i},y_{i})<\varepsilon-\delta$, and define $\nu=%
{\textstyle\sum_{i=1}^{k}}
\nu_{\delta}(\left\{  x_{i}\right\}  )\delta_{y_{i}}$. Now, $\nu\in A_{\sharp
}$, and because of the choice of $x_{i}$ and $y_{i}$ it is $\Delta(\nu
_{\delta},\nu)\leq\varepsilon-\delta$, hence by triangle inequality
$\Delta(\mu,\nu)<\varepsilon$.

(ii) \textit{Claim:\ }$\mathcal{V\subseteq}U_{\sharp}$. Choose any $\mu
\in\mathcal{V}$, and by definition of $\mathcal{V}$ we can find $\nu\in$
$A_{\sharp}$ such that $\Delta(\mu,\nu)<\varepsilon$. We choose $\delta>0$
such that $\Delta(\mu,\nu)\leq\varepsilon-2\delta$. We find $f,g:I\rightarrow
X$, such that $\mu=f_{\sharp}\lambda$, $\nu=g_{\sharp}\lambda$, and such that
$D(f,g)\leq\Delta(\mu,\nu)+\delta\leq\varepsilon-\delta$. Applying Lemma
\ref{l:simple}, we find $\widetilde{f}$, $\widetilde{g}$, as in the Lemma, and
then $D(\widetilde{f},\widetilde{g})\leq\varepsilon-\delta$. For any $x\in
$supp$(\mu)$, we can find $t\in I\,\ $such that $d(x,\widetilde{f}(t))<\delta
$, and since by definition $d(\widetilde{f}(t),\widetilde{g}(t))\leq
D(\widetilde{f},\widetilde{g})\leq\varepsilon-\delta$, we get
that\ $d(x,\widetilde{g}(t))<\varepsilon$. By construction of $\widetilde{g}$,
$\widetilde{g}(t)\in$supp$(\nu)\subseteq A $, therefore $x\in U$.

$\Leftarrow:$ It now follows from uniqueness of $\varepsilon$-neighborhood.
\end{proof}

Lemma \ref{l:help} is not true for weak* topology, or uniform topology.

\begin{corollary}
A set $A\subseteq X$ is open (respectively closed), if and only if $A_{\sharp
}$ is open (respectively closed).
\end{corollary}

\begin{proof}
Lemma \ref{l:help} implies that $A$ is open if and only if $A_{\sharp}$ is open.

Assume that $A$ is closed, and choose any convergent sequence of measures
$\mu_{n}\in A_{\sharp}$, converging to $\mu$. Proposition \ref{p:basic}, (ii)
now implies that supp$(\mu)\subseteq A$, hence $\mu\in A_{\sharp}$ and
$A_{\sharp}$ is closed. Now, assume that $A_{\sharp}$ is closed, and let
$x_{n}$ be any convergent sequence in $A$, converging to $x$. Now by
definition of d-convergence, $\delta_{x_{n}}$ converges to $\delta_{x}$, and
since $A_{\sharp}$ is closed, $\delta_{x}\in A_{\sharp}$. By definition $x\in
A$, so $A$ is closed.
\end{proof}

\begin{corollary}
Let $f:X\rightarrow X$ be continuous. Then a set $A\subseteq X$ is closed and
$f$-invariant if and only if $A_{\sharp}$ is closed and $f_{\sharp}$-invariant.
\end{corollary}

\begin{definition}
\textbf{Lyapunov stability. }Given a continuous function $f$ on $X$, we say
that a closed invariant set $A$ of the dynamical system $f$ is Lyapunov
stable, if for each $\varepsilon>0$, there exists $\delta>0$, such that if
$U$, $V$ are $\varepsilon$, $\delta$ neighborhoods of $A$ respectively, then
for all $n\geq0$, $n\in%
\mathbb{N}
$, $f^{n}(V)\subseteq U$.
\end{definition}

\begin{proposition}
Suppose $f$ is a continuous function on $X$. A closed invariant set $A$ is
Lyapunov stable with respect to $f$, if and only if $A_{\sharp}$ is Lyapunov
stable with respect to $f_{\sharp}$.
\end{proposition}

\begin{proof}
Suppose $A$ is Lyapunov stable, i.e. for a given $\varepsilon>0$,
$f^{n}(V)\subseteq U$ for some $\delta>0$ and all $n\geq0$, with $U$, $V$
respectively $\varepsilon$, $\delta$ neighborhoods of $A$. By definition,
$f^{n}(V)\subseteq U$ is equivalent to $f^{n}(V)_{\sharp}\subseteq U_{\sharp}%
$, and because of Lemma \ref{l:seemstrivial} and $f^{n}(V)_{\sharp}=f_{\sharp
}^{n}(V_{\sharp})$, this is equivalent to\ $f_{\sharp}^{n}(V_{\sharp
})\subseteq U_{\sharp}$. Lemma \ref{l:help} states that $U$, $V$ are
respectively $\varepsilon$, $\delta$ neighborhoods of $A$ if and only if
$U_{\sharp}$, $V_{\sharp}$ are respectively $\varepsilon$, $\delta$
neighborhoods of $A_{\sharp}$, which completes the proof.
\end{proof}

\begin{definition}
\textbf{Asymptotic stability. }Given a continuous function $f$ on $X$, we say
that a closed invariant set $A$ of a dynamical system $f$ is asymptotically
stable, if there exists $\varepsilon>0$, such that for each $x\in U$, where
$U$ is the $\varepsilon$-neighborhood of $A$, $\lim_{n\rightarrow\infty
}d(f^{n}(x),A)=0.$
\end{definition}

\begin{proposition}
Suppose $f$ is a continuous function on $X$. A closed invariant set $A$ is
asymptotically stable with respect to $f$, if and only if $A_{\sharp}$ is
asymptotically stable with respect to $f_{\sharp}$.
\end{proposition}

\begin{proof}
$\Longrightarrow:$ Suppose $A$ is asymptotically stable. Let $\varepsilon>0 $,
$U$,\ be as in the definition of the asymptotic stability, and choose any
$\delta>0$. Now, because of compactness of $X$, there is $n_{0}$ such that for
all $n\geq n_{0}$, $f^{n}(U)\subseteq V$, where $V$ is the $\delta
$-neighborhood of $A$. This, and Lemma \ref{l:help}, imply that for $n\geq
n_{0}$, and any measure $\mu\in$ $U_{\sharp}$, $f_{\sharp}^{n}(\mu)$ is in the
$\delta$-neighborhood of $A_{\sharp}$. Since $\delta$ was arbitrary,
$A_{\sharp}$ is indeed asymptotically stable$.$

$\Leftarrow:$ Assume that $A_{\sharp}$ is asymptotically stable, and choose
$\varepsilon>0$, $U_{\sharp}$, as in the definition of asymptotic stability,
where because of Lemma \ref{l:help}, $U$ is the $\varepsilon$-neighborhood of
$A$. By definition, if $x\in U$, then $\delta_{x}\in U_{\sharp}$, and then
$\Delta(f_{\sharp}^{n}(\delta_{x}),A_{\sharp})\rightarrow0\,$\ as
$n\rightarrow\infty$. That, and the relation $f_{\sharp}^{n}(\delta
_{x})=\delta_{f^{n}(x)}$ now imply that $d(f^{n}(x),A)\rightarrow0$ as
$n\rightarrow\infty$, which proves that $A$ is asymptotically stable.
\end{proof}

Somewhat stronger property than asymptotic stability is that of an attractor.

\begin{definition}
\textbf{Attractor. }Given a continuous function $f$ on $X$, a closed invariant
set $A$ of the dynamical system $f$ is an attractor, if there exist
$\varepsilon>0$, and $N>0$, such that if $U$ is $\varepsilon$-neighborhood of
$A$, then $f^{N}(U)\subseteq U$, and $A=\cap_{n=1}^{\infty}f^{n}(U)$.
\end{definition}

(There are various definitions of attractor in the literature;\ we have chosen
the definition from \cite{Katok:95}.)

\begin{proposition}
Suppose $f$ is a continuous function on $X$. A closed invariant set
$A\subseteq X$ is an attractor with respect to $f$, if and only if $A_{\sharp
}$ is an attractor with respect to $f_{\sharp}$.
\end{proposition}

\begin{proof}
The Proposition follows from the following equivalences: (I) Lemma
\ref{l:help} implies that $U_{\sharp}$ is the $\varepsilon$-neighborhood of
$A_{\sharp}$, if and only if $U$ is the $\varepsilon$-neighborhood of $A$;
(II)\ Lemma \ref{l:elementary}, (ii)\ and Lemma \ref{l:seemstrivial} imply
that $f^{N}(U)\subseteq U$ if and only if $f_{\sharp}^{N}(U_{\sharp})\subseteq
U_{\sharp}$; and (III) Lemma \ref{l:elementary}, (iv)\ and Lemma
\ref{l:seemstrivial} imply that $\cap_{k=1}^{\infty}f_{\sharp}^{n}(U_{\sharp
})=(\cap_{n=1}^{\infty}f^{n}(U))_{\sharp}$. That, and Lemma \ref{l:elementary}%
, (i), now imply that $A=\cap_{n=1}^{\infty}f^{n}(U)$ if and only if
$A_{\sharp}=\cap_{n=1}^{\infty}f_{\sharp}^{n}(U_{\sharp})$.
\end{proof}

An analogous claim holds for repellers, sinks and for sources.

\begin{definition}
\textbf{Exponential stability. }Given a continuous function $f$ on $X$, a
closed invariant set $A$ of the dynamical system $f$ is exponentially stable,
if there exist constants $C$, $\lambda>0$, and $\varepsilon>0$, such that for
each $0<\delta<\varepsilon$, if $U$ is a closed $\delta$ neighborhood of $A$,
then, $d_{H}(A,f^{n}(U))\leq Ce^{-\lambda n}d_{H}(A,U)$.
\end{definition}

We now introduce notation related to the Hausdorff metric induced by $\Delta$
on closed subsets of $P(X).$ Given $\mu\in P(X)$, we set $\Delta
(\mu,\mathcal{A})=\inf\{\Delta(\mu,\nu),\nu\in\mathcal{A}\}$, $\Delta
(\mathcal{A},\mathcal{B})=\sup\{\Delta(\mu,\mathcal{B})\,|\,\mu\in
\mathcal{A}\}$, and then $\Delta_{H}(\mathcal{A},\mathcal{B})=\max
\{\Delta(\mathcal{A},\mathcal{B}),\Delta(\mathcal{B},\mathcal{A})\}$ is the
Hausdorff metric.

\begin{lemma}
\label{l:haushelp}Let $U,V\subseteq X$ be closed sets. Then $d_{H}%
(U,V)=\Delta_{H}(U_{\sharp},V_{\sharp})$.
\end{lemma}

\begin{proof}
Lemma \ref{l:simple} implies that $\Delta(\mu,\nu)\leq\sup_{x\in
\text{supp}(\mu)}d(x,$supp$(\nu))$, hence $\Delta(\mu,V_{\sharp})\leq
\sup_{x\in\text{supp}(\mu)}d(x,V)$.\ We now deduce that for $\mu\in U_{\sharp
}$, $\Delta(\mu,V_{\sharp})\leq\sup_{x\in U}d(x,V)=d(U,V)$, and by taking
supremum over $\mu\in U_{\sharp}$, we get $\Delta(U_{\sharp},V_{\sharp})\leq
d(U,V)$.

The definition of $\Delta$ and Lemma \ref{l:simple} easily imply that
$\Delta(\delta_{x},V_{\sharp})=d(x,V)$, hence $\Delta(U_{\sharp},V_{\sharp
})\geq\sup_{x\in U}\Delta(\delta_{x},V_{\sharp})=\sup_{x\in U}\Delta
(x,V)=d(U,V)$. We conclude that $\Delta(U_{\sharp},V_{\sharp})=d(U,V)$, and
similarly $\Delta(V_{\sharp},U_{\sharp})=d(V,U)$, which completes the proof.
\end{proof}

\begin{proposition}
Suppose $f$ is a continuous function on $X$. A closed invariant set $A$ is
exponentially stable with respect to $f$, if and only if $A_{\sharp}$ is
exponentially stable with respect to $f_{\sharp}$.
\end{proposition}

\begin{proof}
It follows directly from the definition of exponential stability, Lemma
\ref{l:help} and Lemma \ref{l:haushelp}.
\end{proof}

A similar claim can be also proven for Nekhoroshev stability (see e.g.
\cite{Nekhoroshev:77}).

We now give an example which shows that the claims above do not hold for
uniform or weak* topology. In other words, we show that uniform or weak*
topology on $P(X)$ are not the right topologies for generalizing notions of
stability to spaces of measures.

\begin{example}
Assume that $f$ is a dynamical system with one attracting sink $x$, and one
source $y$. Such a dynamical system can be constructed on a, say, 2-sphere,
with the sink and the source being the poles. We denote by $\delta_{x}$ and
$\delta_{y}$ the probability measures concentrated on $x$, $y$ respectively,
and we define $\mu_{\varepsilon}=(1-\varepsilon)\delta_{x}+\varepsilon
\delta_{y}$, for a given $\varepsilon>0\,$. Now, for small $\varepsilon>0$,
$\mu_{\varepsilon}$ is arbitrarily u-close and w-close to $\delta_{x}$ (but
not d-close!). However, $\mu_{\varepsilon}$ is a $f_{\sharp}$-fixed point.\ We
conclude that $\{\delta_{x}\}=\{x\}_{\sharp}$ is neither an attractor, nor
asymptotically stable set for $f_{\sharp}$ in weak* or uniform topology on
$P(X)$.
\end{example}

Now, we naturally generalize the notions of stability to $f$-invariant
measures, or more generally to $f_{\sharp}$-invariant sets of measures. (Note
that if $\mu$ is an $f$-invariant measure, then $\mu$ is $f_{\sharp} $-fixed
point, and $\{\mu\}$ is a $f_{\sharp}$-invariant closed set.)

\begin{definition}
Suppose $f$ is a continuous function on $X$.

We say that a $f$-invariant measure $\mu$ is Lyapunov stable, asymptotically
stable, an attractor, or exponentially stable, if $\{\mu\}$ is Lyapunov
stable, asymptotically stable, an attractor, or exponentially stable
respectively, with respect to $f_{\sharp}$ and the dynamical topology on
$P(X)$.

More generally, if $\mathcal{A}\mathbb{\ }$is a $f_{\sharp}$-invariant,
d-closed set of measures, we say that it is Lyapunov stable, asymptotically
stable, an attractor, or exponentially stable, if it is so with respect to
$f_{\sharp}$ and the dynamical topology on $P(X)$.
\end{definition}

An information on stability of invariant measures gives much more information
on dynamics in a neighborhood of a set (which can be the support of an
invariant measure), then the information on stability of invariant sets. The
following example illustrates that claim.

\begin{example}
Let $X=%
\mathbb{R}
^{2}/%
\mathbb{Z}
^{2}$ be a 2-torus, and we define a function $f(x,y)=(x^{\prime},y^{\prime})$
with
\begin{align*}
x^{\prime}  & =x+y,\\
y^{\prime}  & =y,
\end{align*}
(a standard map with $k=0$). Let $A$ be the circle $y=0$. Then $A$ is a closed
invariant set (all points on $A$ are fixed), and also Lyapunov stable. Let
$\lambda$ be the Lebesgue measure on $A$, and $\nu$ any probability measure
supported on $A$, different from $\lambda$. Both $\lambda$, $\nu$ are
$f$-invariant, and supported on a Lyapunov stable set $A$. Now, one can check
that $\lambda$ is Lyapunov stable, and $\nu$ is not. This reflects the fact
that the rotation in the vicinity of $A$ is "with uniform speed".\ Similar
examples could be constructed for attractors and exponentially stable sets.
\end{example}

\section{The optimal transport problem}

In this section we define the {\textbf{$\infty$-}}optimal transport problem,
and show that the metric on the set of measures induced by the
{\textbf{$\infty$-}}optimal transport problem is equal to the dynamical
$\Delta$ metric. We then deduce that the {\textbf{$\infty$-}}optimal transport
problem generates a different structure on the set of of measures than the
$p${\textbf{-}}optimal transport problem for any $1\leq p<\infty$.

\begin{definition}
Let $X$ be a metric space with a metric $d$, and $\mu$, $\nu$ be two Borel
probability measures on $X$.

The set of all transports $T(\mu,\nu)$ of measures $\mu,\nu$ is the set of all
Borel probability measures $\gamma$ on $X\times X$, such that $\pi_{1\sharp
}\gamma=\mu$, $\pi_{2\sharp}\gamma=\nu$, where $\pi_{1}$, $\pi_{2}$ are
projections of $X\times X$ to the first, resp. second variable.

Distance of measures $\mu,\nu$ with respect to a transport $\gamma\in
T(\mu,\nu)$ is defined with
\[
\Delta_{\gamma}(\mu,\nu)=\inf\{\sup\{d(x,y)\,|(x,y)\in A\}\,|\,A\subseteq
X\times X\text{ measurable, and }\gamma(A)=1\}\text{.}%
\]

The $\infty$-Wasserstein distance of two measures $\mu,\nu$ is defined with
\begin{equation}
\Delta_{\infty}(\mu,\nu)=\inf\{\Delta_{\gamma}(\mu,\nu)\,|\,\gamma\in
T(\mu,\nu)\}\text{. }\label{iw}%
\end{equation}

If the minimum in (\ref{iw})\ is attained, then any measure $\gamma$ for which
$\Delta_{\gamma}(\mu,\nu)=\Delta_{\infty}(\mu,\nu)$ is called a solution of
the optimal transport problem with respect to the measures $\mu,\nu$.
\end{definition}

Now we prove that the $\infty$-Wasserstein distance is the same as the
dynamical metric $\Delta$.

\begin{proposition}
Given two probability measures $\mu,\nu$, then $\Delta_{\infty}(\mu
,\nu)=\Delta(\mu,\nu)$.
\end{proposition}

\begin{proof}
\textit{Claim:\ }$\Delta_{\infty}(\mu,\nu)\leq\Delta(\mu,\nu)$. Choose any
measurable $f,g:I\rightarrow X$, $f_{\sharp}\lambda=\mu$, $\ g_{\sharp}%
\lambda=\nu$. Then we define $\gamma:=(f\times g)_{\sharp}\lambda$, where
$(f\times g)(x):=(f(x),g(x))$. The assumptions imply that $\gamma\in T(\mu
,\nu)$, and then
\[
\Delta_{\infty}(\mu,\nu)\leq\Delta_{\gamma}(\mu,\nu)\leq\sup
\{d(x,y)\,|\,(x,y)\in(f\times g)(I)\}=D(f,g)\text{.}%
\]

Now, it is sufficient to take infimum of the right hand side of the equation
above over all $f,g$, such that $f_{\sharp}\lambda=\mu$, $\ g_{\sharp}%
\lambda=\nu$.\smallskip

\textit{Claim:\ }$\Delta(\mu,\nu)\leq\Delta_{\infty}(\mu,\nu)$. Suppose
$\gamma\in T(\mu,\nu)$. By applying Theorem \ref{t:big}, we can find a
function $h:I\rightarrow X\times X$, such that $\gamma=h_{\sharp}\lambda$. We
define \thinspace$f=\pi_{1}\circ h$, $g=\pi_{2}\circ h$, and then
$\mu=f_{\sharp}\lambda$, $\nu=g_{\sharp}\lambda$. Let $A\subseteq X\times X$
be any set such that $\gamma(A)=1$, and let $J=h^{-1}(A)$. Now definitions and
Lemma \ref{l:topA} imply that
\begin{align*}
\Delta(\mu,\nu)  & \leq D_{J}(f,g)=\sup\{d(x,y)\,|\,(x,y)\in h(J)\}\\
& =\sup\{d(x,y)\,|\,(x,y)\in A\}\text{.}%
\end{align*}

The proof is completed by taking the infimum of the right hand side of the
equation above over all $A,\gamma$.
\end{proof}

\begin{corollary}
Assume that $X$ has at least two elements. Then $\infty$-Wasserstein metric
$\Delta_{\infty}=\Delta$ is neither uniformly nor topologically equivalent to
any of the p-Wasserstein metrices $W_{p}$ for $1\leq p<\infty$.
\end{corollary}

\begin{proof}
For any $p$, $1\leq p<\infty$, the metric $W_{p}$ is uniformly equivalent to
the Prokhorov metric on $P(X)$ (see e.g. \cite{Gibbs:02}), hence topologically
equivalent to the weak* topology. Corollary \ref{c:horprop},\ (i) implies the claim.
\end{proof}

We close this Section with a proof that the {\textbf{$\infty$-}}optimal
transport problem\thinspace has a solution.

\begin{proposition}
There exists a measure $\gamma_{0}\in T(\mu,\nu)$, for which the minimum of
the expression (\ref{iw}) is attained.
\end{proposition}

\begin{proof}
For given $\mu,\nu\in P(X)$, we first prove that the map $\gamma
\longmapsto\Delta_{\gamma}(\mu,\nu)$ is w-lower semi continuous. Since
$d:X\times X\rightarrow%
\mathbb{R}
$ is continuous, $\{\sup\{d(x,y)\,|(x,y)\in A\}=\{\sup\{d(x,y)\,|(x,y)\in
Cl(A)\}$, and
\begin{equation}
\Delta_{\gamma}(\mu,\nu)=\{\sup\{d(x,y)\,|\,(x,y)\in\text{supp }%
\gamma\}.\label{e:tran1}%
\end{equation}
Choose any sequence of measures $\gamma_{n}\rightarrow\gamma_{0}$, convergent
in w-topology, as $n\rightarrow\infty$. Now for any $\varepsilon>0$, and any
$(x,y)\in$ supp $\gamma_{0}$, let $B_{\varepsilon}$ be an\ open $\varepsilon
$-ball around $(x,y)$, and then by definition of the support $\gamma
_{0}(B_{\varepsilon})\,>0$. Since $\lim\inf\gamma_{n}(B_{\varepsilon}%
)\geq\gamma_{0}(B_{\varepsilon})\,>0$, (\ref{e:tran1}) implies that there is
$n_{0}$ such that for all $n\geq n_{0}$, $\gamma_{n}(B_{\varepsilon})>0$. For
each $n\geq n_{0}$, there is a point $(x^{\prime},y^{\prime})\in$ supp
$\gamma_{n}\cap B_{\varepsilon}$, and then because of (\ref{e:tran1}),
$\Delta_{\gamma_{n}}(\mu,\nu)\geq d(x^{\prime},y^{\prime})\geq
d(x,y)-2\varepsilon$. Since $\varepsilon$ was arbitrary, we conclude that
$\lim\inf\Delta_{\gamma_{n}}(\mu,\nu)\geq d(x,y)$. Since $(x,y)$ $\in$ supp
$\gamma_{0}$ was arbitrary, we get $\lim\inf\Delta_{\gamma_{n}}(\mu,\nu
)\geq\Delta_{\gamma_{0}}(\mu,\nu)$.

Suppose now $\gamma_{n}\in T(\mu,\nu)$ is a sequence of transports such that
$\Delta_{\gamma_{n}}(\mu,\nu)\leq\Delta_{\infty}(\mu,\nu)+1/n$. Now
$\gamma_{n}$ has a w-convergent subseqence, which converges to another
transport $\gamma_{0}$. By definition $\Delta_{\gamma_{0}}(\mu,\nu)\geq
\Delta_{\infty}(\mu,\nu)$, and because of w-lower semicontinuity,
$\Delta_{\gamma_{0}}(\mu,\nu)\leq\Delta_{\infty}(\mu,\nu)\,$, so $\gamma_{0}$
is the required measure.
\end{proof}

\begin{corollary}
For given probability measures $\mu,\nu$, there exist measurable functions
$f,g:I\rightarrow X$, $f_{\sharp}\lambda=\mu$, $g_{\sharp}\lambda=\nu$, such
that $\Delta(\mu,\nu)=D(f,g)$.
\end{corollary}

\begin{proof}
If $\gamma_{0}\in T(\mu,\nu)$ is the measure for which (\ref{iw}) is minimal,
and $h:I\rightarrow X\times X$, $h$ a measurable function such that
$\gamma_{0}=h_{\sharp}\lambda$, then $f=\pi_{1}\circ h$, $g=\pi_{2}\circ h$
are required functions.
\end{proof}

\section{Appendix:\ The proof of Lemma \ref{l:claim2}}

In this Appendix we prove the fact from the proof of Theorem \ref{t:char},
which is essentially combinatorial in character. Assume that an integer $m$,
any family of measurable subsets $(B_{1},B_{2},...,B_{m})$ of $X$, and a
measure $\xi$ are given. We first introduce some notation and definitions. Let
$\mathcal{P}(\{1,...,m\})$ be the set of all subsets of $\{1,...,m\}$, and for
a nonempty $\varphi\in\mathcal{P}(\{1,...,m\})$,
\[
B_{\varphi}:=B_{i_{1}}\cap B_{i_{2}}\cap...\cap B_{i_{k}}\cap B_{j_{1}}%
^{C}\cap B_{j_{2}}^{C}\cap...\cap B_{j_{m-k}}^{C},
\]
where $\varphi=\{i_{1},i_{2},...,i_{k}\}$, $\varphi^{C}=\{j_{1},j_{2}%
,...,j_{m-k}\}$. Let $|\varphi|$ denote the cardinal number of $\varphi$. We
define the \textit{arrangement} of $(B_{1},B_{2},...,B_{m})$ to be $\rho
(B_{1},B_{2},...,B_{m})=\sum_{k=1}^{m}\rho_{k}$, where%
\[
\rho_{k}=|\{\varphi\subseteq\{1,...,m\}:|\varphi|=k,\;\text{and }%
\xi(B_{\varphi})>0\}|
\]
(The arrangement $\rho$ depends also on the measure $\xi$, which we omit from
the argument of $\rho$ because it is always clear which measure is being
considered). The $k^{th}$ arrangement $\rho_{k}$ shows how many sets of
intersections of exactly $k$ sets $B_{i}$ have positive measure $\xi$. The
arrangement\ $\rho$ is the number of such sets $B_{\varphi}$ for any $k>0$,
such that they have positive measure $\xi$, and then%
\begin{equation}
\rho\leq2^{m}-1.\label{arrgm}%
\end{equation}

Now we prove Lemma \ref{l:claim2}.

\begin{proof}
Let $\xi$\ be a measure on $X$\ (positive, not necessarily a normed one),
$x_{1},x_{2},...,x_{m}$, nonnegative real numbers, and $B_{1},B_{2},...,B_{m}%
$\ measurable subsets of $X$, such that (\ref{com1}), (\ref{com2}) hold.

We prove the claim inductively, with\ respect to two integers $m$, $\rho$,
where $m$ is the number of sets $B_{i}$, $i=1,...,m$, and $\rho$ their
arrangement $\rho=\rho(B_{1},B_{2},...,B_{m})$.

The basis of the induction is the case $m=1$, $\rho=1$. Then $\nu_{1}:=\xi$
clearly satisfies (\ref{com3.2}), (\ref{com3.3}), (\ref{com3.4}).\medskip

Now, for a given $m$, $\rho$, assume that the claim is true for any family of
measurable sets $(B_{i}^{\prime})_{i=1,...,m^{\prime}}$, where $m^{\prime}\leq
m$, $\rho^{\prime}=\rho(B_{1}^{\prime},....,B_{m^{\prime}}^{\prime})\leq\rho$,
and either $m^{\prime}<m$ or $\rho^{\prime}<\rho$. Choose a family of
measurable sets $B_{1},B_{2},...,B_{m}$, with the arrangement $\rho$, and
assume that $x_{1},x_{2},...,x_{m}$ are nonnegative real numbers satisfying
(\ref{com1}), (\ref{com2}). We analyze three cases:\medskip

\textit{Case 1:\ There exist }$\iota=\{i_{1},i_{2},...,i_{k}\}$, $1\leq k\leq
m-1$, \textit{such that there is equality in (\ref{com1}), namely that }%
\[
\xi(B_{i_{1}}\cup B_{i_{2}}\cup...\cup B_{i_{k}})=x_{i_{1}}+x_{i_{2}%
}+...+x_{i_{k}}\text{.}%
\]
Let $\{i_{k+1},i_{k+2},...,i_{m}\}=\{i_{1},i_{2},...,i_{k}\}^{C}$. We define%
\begin{align*}
B_{j}^{\prime}  & =B_{i_{j}}\\
x_{j}^{\prime}  & =x_{i_{j}}\\
\xi^{\prime}  & =\xi\,|\,(B_{i_{1}}\cup B_{i_{2}}\cup...\cup B_{i_{k}})\\
B_{j}^{\prime\prime}  & =B_{i_{k+j}}\backslash\,(B_{i_{1}}\cup B_{i_{2}}%
\cup...\cup B_{i_{k}})\\
x_{j}^{\prime\prime}  & =x_{i_{k+j}}\\
\xi^{\prime\prime}  & =\xi\,|\,(B_{i_{1}}\cup B_{i_{2}}\cup...\cup B_{i_{k}%
})^{C}\text{,}%
\end{align*}
where $\xi\,|\,A$ is the measure $\xi\,|\,A(Y)=\xi(A\cap Y)$. It is easy to
check that $(B_{j}^{\prime})$, $(x_{j}^{\prime})$, $\xi^{\prime}$;\ and also
$(B_{j}^{\prime\prime})$, $(x_{j}^{\prime\prime})$, $\xi^{\prime\prime}$
satisfy (\ref{com1}), and (\ref{com2}). By applying the inductive assumption,
we can find $\nu_{i}^{\prime},\,i=1,...,k$, and $\nu_{i}^{\prime\prime
},\,i=1,...,m-k$, satisfying (\ref{com3.2}), (\ref{com3.3}), (\ref{com3.4}).
We define
\[
\nu_{j}=\left\{
\begin{array}
[c]{ccc}%
v_{j}^{\prime} & \text{if} & j=1,...,k,\\
\nu_{j-k}^{\prime\prime} & \text{if} & j=k+1,...,m,
\end{array}
\right.
\]
and then $\nu_{1},....,\nu_{m}$ satisfy (\ref{com3.2}), (\ref{com3.3}),
(\ref{com3.4}).\medskip

\textit{Case 2:\ There exists }$k$\textit{, }$1\leq k\leq m$\textit{, such
that }$x_{k}=0$\textit{. \medskip}

We define $\nu_{k}=0$. We set $B_{j}^{\prime}=B_{j}$, $x_{j}^{\prime}=x_{j}$,
for $j=1,...,k-1$, and $B_{j}^{\prime}=B_{j+1}$, $x_{j}^{\prime}=x_{j+1}$, for
$j=k+1,...,m$. Then $(B_{j}^{\prime})$, $(x_{j}^{\prime})$, $\xi$ satisfy
(\ref{com1}), and (\ref{com2}), and by applying the inductive assumption we
find $\nu_{i}^{\prime},\,i=1,...,m-1$, satisfying (\ref{com3.2}),
(\ref{com3.3}), (\ref{com3.4}). Then we set $\nu_{j}=\nu_{j}^{\prime}$, for
$j=1,...,k-1$, and $\nu_{j}=\nu_{j-1}^{\prime}$, for $j=k+1,...,m$, and so
prove the claim.\medskip

\textit{Case 3:\ For all }$1\leq k\leq m-1$\textit{, and all subsets }%
$\{i_{1},i_{2},...,i_{k}\}\subset\{1,...,m\}$\textit{, }%
\[
\xi(B_{i_{1}}\cup B_{i_{2}}\cup...\cup B_{i_{k}})>x_{i_{1}}+x_{i_{2}%
}+...+x_{i_{k}}\text{,}%
\]
\textit{and also for all }$1\leq k\leq m$\textit{, }$x_{k}>0$\textit{.
\medskip}

We rewrite the relations (\ref{com1}), using the introduced notation:\ for all
$1\leq i_{1}<i_{2}<...<i_{k}\leq m$, $1\leq k\leq m-1$,%
\begin{equation}
\xi(B_{i_{1}}\cup B_{i_{2}}\cup...\cup B_{i_{k}})=%
{\displaystyle\sum\limits_{\varphi\cap\left\{  i_{1},i_{2},...,i_{k}\right\}
\neq\emptyset}}
\xi(B_{\varphi})>x_{i_{1}}+x_{i_{2}}+...+x_{i_{k}}\text{,}\label{vozi1}%
\end{equation}
where $\varphi$ is any nonempty $\varphi\subseteq\{1,...,m\}$. Let
\[
\delta_{\{i_{1},i_{2},...,i_{k}\}}=\xi(B_{i_{1}}\cup B_{i_{2}}\cup...\cup
B_{i_{k}})-x_{i_{1}}-x_{i_{2}}-...-x_{i_{k}}\text{.}%
\]
We now define
\begin{equation}
\delta=\min\{\delta_{\varphi}\text{, }1\leq|\varphi|\leq m-1\}\text{,
}\label{mini}%
\end{equation}
and let $\iota=\{j_{1},...,j_{k}\}$, $1\leq k\leq m-1$, be the set for which
(\ref{mini}) is minimal. The assumption of the Case 3 implies that $\delta>0$.
Choose any $\psi$ such that $\xi(B_{\psi})>0$. (Note that the assumptions of
the Case 3 do not apply that for all nonempty $\varphi$, $\xi(B_{\varphi}%
)>0$.) Let $p$ be any $p\in\psi$.

For some $\varepsilon$ (to be chosen later), we define:
\begin{align*}
B_{i}^{\prime}  & =B_{i}\text{,}\\
x_{i}^{\prime}  & =\left\{
\begin{array}
[c]{ccc}%
x_{i}-\varepsilon & \text{if} & i=p,\\
x_{i} & \,\text{if} & i\not =p,
\end{array}
\right. \\
\xi^{\prime}  & =\xi-\frac{\varepsilon}{\xi(B_{\psi})}(\xi\,|\,B_{\psi
})\text{,}%
\end{align*}
and then (\ref{vozi1}) implies that%
\[
\xi^{\prime}(B_{i_{1}}^{\prime}\cup B_{i_{2}}^{\prime}\cup...\cup B_{i_{k}%
}^{\prime})=\left\{
\begin{array}
[c]{ccc}%
\xi(B_{i_{1}}\cup B_{i_{2}}\cup...\cup B_{i_{k}})-\varepsilon & \text{if} &
(i_{1},...,i_{k})\cap\psi\not =\emptyset,\\
\xi(B_{i_{1}}\cup B_{i_{2}}\cup...\cup B_{i_{k}}) & \text{if} & (i_{1}%
,...,i_{k})\cap\psi=\emptyset.
\end{array}
\right.
\]
Let $\varepsilon_{0}$ be the maximal $\varepsilon$, such that for all nonempty
$\{i_{1},i_{2},...,i_{k}\}\subseteq\{1,...,m\}$, $1\leq k\leq m-1$,
\begin{equation}
\xi(B_{i_{1}}^{\prime}\cup B_{i_{2}}^{\prime}\cup...\cup B_{i_{k}}^{\prime
})\geq x_{i_{1}}^{\prime}+x_{i_{2}}^{\prime}+...+x_{i_{k}}^{\prime
}.\label{eqeq}%
\end{equation}
Such $\varepsilon_{0}>0$ exists because of $\delta>0$, and for $\varepsilon
=\varepsilon_{0}$, there is equality in (\ref{eqeq}) for some $\{i_{1}%
,i_{2},...,i_{k}\}$, $1\leq k\leq m-1$. We now fix $\varepsilon$ to be
\[
\varepsilon=\min\{\varepsilon_{0},x_{p},\xi(B_{\psi})\}\text{.}%
\]
The relation (\ref{eqeq}) and the fact that $\varepsilon\leq\varepsilon_{0}$
imply that $(B_{j}^{\prime})$, $(x_{j}^{\prime})$, $\xi^{\prime}$ satisfy
(\ref{com1});\ $\varepsilon\leq x_{p}$ implies that $(x_{i}^{\prime})$,
$i=1,...,m$ are nonnegative;\ and $\varepsilon\leq\xi(B_{\psi})$ implies that
$\xi^{\prime}$ is well defined (i.e. a nonnegative)\ measure.

By exchanging $(B_{j})$, $(x_{j})$, $\xi$ with $(B_{j}^{\prime})$,
$(x_{j}^{\prime})$, $\xi^{\prime}$ in (\ref{com2}) we subtract $\varepsilon$
from both sides, so (\ref{com2}) still holds for $(B_{j}^{\prime})$,
$(x_{j}^{\prime})$, $\xi^{\prime}$. We claim now that by applying the
assumption of the induction, we can construct $\nu_{1}^{\prime},\nu
_{2}^{\prime},...,\nu_{m}^{\prime}$, satisfying (\ref{com3.2}), (\ref{com3.3}%
), (\ref{com3.4}). We again discuss three cases: \medskip

\textit{Case 3.1. }$\varepsilon=\varepsilon_{0}$. Then for some $\{i_{1}%
,i_{2},...,i_{k}\}$, $1\leq k\leq m-1$, there is equality in
\textit{(\ref{com1}),} so the problem reduces to the problem analyzed in the
Case 1.\medskip

\textit{Case 3.2. }$\varepsilon=x_{p}$. Then $x_{p}^{\prime}=0$, so the
problem reduces to the problem analyzed in the Case 2.\medskip

\textit{Case 3.3. }$\varepsilon=\xi(B_{\psi})$. In this case, $\xi^{\prime
}(B_{\psi}^{\prime})=0$, so $\rho(B_{1}^{\prime},B_{2}^{\prime},...,B_{m}%
^{\prime})=\rho(B_{1},B_{2},...,B_{m})-1$, hence we can construct $\nu
_{1}^{\prime},\nu_{2}^{\prime},...,\nu_{m}^{\prime}$ inductively.\medskip

Now we define
\[
\nu_{i}=\left\{
\begin{array}
[c]{ccc}%
\nu_{i}^{\prime}+\frac{\varepsilon}{\xi(B_{\psi})}(\xi\,|\,B_{\psi}) &
\text{if} & i=p,\\
\nu_{i}^{\prime} & \,\text{if} & i\not =p.
\end{array}
\right.
\]
Then the construction implies that $\nu_{i}$, $i=1,...,m$ satisfy
(\ref{com3.2}), (\ref{com3.3}), (\ref{com3.4}).
\end{proof}

\end{document}